\documentclass{amsart}
\usepackage{amsmath}
\usepackage{amsfonts}
\usepackage{amssymb}
\usepackage{latexsym}

\newcommand{\abs}[1]{\vert #1 \vert}

\newcommand{\norm}[1]{\left\Vert #1 \right\Vert}

\newcommand{\bignorm}[1]{\bigl\Vert #1 \bigr\Vert}

\newcommand{\C}{\mathbb{C}}

\newcommand{\R}{\mathbb{R}}

\newcommand{\angles}[1]{\langle #1 \rangle}

\DeclareMathOperator{\im}{Im}
\DeclareMathOperator{\re}{Re}

\newtheorem*{PWtheorem}{Paley-Wiener Theorem}
\newtheorem{theorem}{Theorem}

\newtheorem{lemma}{Lemma}
\newtheorem{corollary}{Corollary}

\theoremstyle{definition}

\theoremstyle{remark}

\title[DKG spatial analyticity]{On the radius of spatial analyticity for the 1d Dirac-Klein-Gordon equations}

\author[S.~Selberg]{Sigmund Selberg}
\author[A.~Tesfahun]{Achenef Tesfahun}

\address{Department of Mathematics, University of Bergen, PO Box 7803, N-5020 Bergen, Norway}
\email{sigmund.selberg@math.uib.no}

\address{Universit\"{a}t Bielefeld, Fakult\"{a}t f\"{u}r Mathematik, Postfach 10 01 31, D-33501 Bielefeld, Germany}
\email{achenef@math.uni-bielefeld.de}



\begin{document}

\begin{abstract} 
We study the well-posedness of the Dirac-Klein-Gordon system in one space dimension with initial data that have an analytic extension to a strip around the real axis. It is proved that the radius of analyticity $\sigma(t)$ of the solutions at time $t$ cannot decay faster than $1/t^4$ as $\abs{t} \to \infty$.
\end{abstract}

\maketitle


\section{Introduction}\label{Intro}

Consider the Dirac-Klein-Gordon equations (DKG) on $\R^{1+1}$,
\begin{equation}
\label{DKG}
\left\{
\begin{aligned}
  \left( -i \gamma^0 \partial_t - i \gamma^1 \partial_x + M \right) \psi &= \phi \psi,
  \\
  \left( \partial_t^2 - \partial_x^2 + m^2 \right) \phi &= \psi^* \gamma^0 \psi,
\end{aligned}
\right.
\qquad (t,x \in \R)
\end{equation}
with initial condition
\begin{equation}\label{Data}
  \psi(0,x) = \psi_0(x),
  \quad
  \phi(0,x) = \phi_0(x), \quad \partial_t \phi(0,x) = \phi_1(0,x).
\end{equation}
Here the unknowns are $\phi \colon \R^{1+1} \to \R$ and $\psi \colon \R^{1+1} \to \C^2$, the latter regarded as a column vector with conjugate transpose $\psi^*$. The masses $M,m \ge 0$ are given constants. The $2 \times 2$ Dirac matrices $\gamma^0,\gamma^1$ should satisfy $\gamma^0\gamma^1 + \gamma^1\gamma^0 = 0$, $(\gamma^0)^2 = I$, $(\gamma^1)^2 = -I$, $(\gamma^0)^* = \gamma^0$ and $(\gamma^1)^* = - \gamma^1$; we will work with the representation
\[
  \gamma^0 = \left( \begin{matrix}
      0 & 1 \\
      1 & 0 \\
  \end{matrix} \right),
  \quad
  \gamma^1 = \left( \begin{matrix}
      0 & -1 \\
      1 & 0 \\
  \end{matrix} \right).
\]

The well-posedness of this Cauchy problem with data in the family of Sobolev spaces $H^s = (1-\partial_x^2)^{-s/2}L^2(\R)$, $s \in \R$, has been intensively studied; see \cite{Chadam:1973, Bournaveas:2000, Bachelot:2006, Bournaveas:2006, Pecher:2006, Machihara:2007, Selberg:2008, Selberg:2007, Pecher:2008, Tesfahun:2009, Machihara:2010, Selberg:2010, Candy:2012}. Local well-posedness holds for data
\begin{equation}\label{SobolevData}
  (\psi_0,\phi_0,\phi_1) \in H^s(\R;\C^2) \times H^r(\R;\R) \times H^{r-1}(\R;\R)
\end{equation}
with $s > -1/2$ and $\abs{s} \le r \le s+1$; see \cite{Machihara:2010}, where it is also proved that this is the optimal result, in the sense that for other $(r,s)$ one either has ill-posedness or the solution map (if it exists) is not regular. Moreover, when $s \ge 0$ there is conservation of charge,
\[
  \norm{\psi(t)}_{L^2} = \norm{\psi_0}_{L^2},
\]
implying that the solutions extend globally when $0 \le s \le r \le s+1$, and by propagation of higher regularity the global solution is $C^\infty$ if the data are $C^\infty$.

While the well-posedness in Sobolev spaces is well-understood, much less is known concerning spatial analyticity of the solutions to the above Cauchy problem, and this is what motivates the present paper.

On the one hand, local propagation of analyticity for nonlinear hyperbolic systems has been studied by Alinhac and M\'etivier \cite{Alinhac1984} and Jannelli \cite{Jannelli:1986}, and in particular this general theory implies that if the data \eqref{SobolevData} (with $s,r$ sufficiently large) are analytic on the real line, then the the same is true of the solution $(\psi,\phi,\partial_t \phi)(t)$ to \eqref{DKG} for all times $t$. The local theory does not give any information about the radius of analyticity, however.

On the other hand, one can consider the situation where a uniform radius of analyticity on the real line is assumed for the initial data, so there is a holomorphic extension to a strip $\{ x+iy \colon \abs{y} < \sigma_0\}$ for some $\sigma_0 > 0$. One may then ask whether this property persists for all later times $t$, but with a possibly smaller and shrinking radius of analyticity $\sigma(t) > 0$. This type of question was introduced in an abstract setting of nonlinear evolutionary PDE by Kato and Masuda \cite{Kato:1986}, who showed in particular that for the Korteweg-de Vries equation (KdV) the radius of analyticity $\sigma(t)$ can decay to zero at most at a super-exponential rate. A similar rate of decay for semilinear symmetric hyperbolic systems has been proved recently by Cappiello, D'Ancona and Nicola \cite{Cappiello:2014}. An algebraic rate of decay for KdV was shown by Bona and Kalisch \cite{Bona2005}. Panizzi \cite{Panizzi2012} has obtained an algebraic rate for nonlinear Klein-Gordon equations. In this paper our aim is to obtain an algebraic rate for the DKG system.

We use the following spaces of Gevrey type. For $\sigma \ge 0$ and $s \in \R$, let $G^{\sigma,s}$ be the Banach space with norm
\[
  \norm{f}_{G^{\sigma,s}} = \bignorm{e^{\sigma\abs{\xi}} \angles{\xi}^s \widehat f(\xi)}_{L^2_\xi},
\]
where $\widehat f(\xi) = \int_{\R} e^{- i x\xi} f(x) \, dx$ is the Fourier transform and $\angles{\xi} = (1+\abs{\xi}^2)^{1/2}$. So for $\sigma > 0$ we have $G^{\sigma,s} = \{ f \in L^2 \colon e^{\sigma\abs{\cdot}} \angles{\cdot}^s \widehat f \in L^2 \}$ and for $\sigma = 0$ we recover the Sobolev space $H^s = G^{0,s}$ with norm $\norm{f}_{H^s} = \bignorm{\angles{\xi}^s \widehat f(\xi)}_{L^2_\xi}$.

Observe the embeddings
\begin{alignat}{2}
  \label{Inclusion1}
  G^{\sigma,s} &\subset H^{s'}& \quad &\text{for $0 < \sigma$ and $s,s' \in \R$},
  \\
  \label{Inclusion2}
  G^{\sigma,s} &\subset G^{\sigma',s'}& \quad &\text{for $0 < \sigma' < \sigma$ and $s,s' \in \R$}.
\end{alignat}
Every function in $G^{\sigma,s}$ with $\sigma > 0$ has an analytic extension to the strip
\[
  S_\sigma = \left\{ x+iy \colon x,y \in \R, \;\abs{y} < \sigma \right\}.
\]

\begin{PWtheorem}
Let $\sigma > 0$, $s \in \R$. The following are equivalent:
\begin{enumerate}
\item $f \in G^{\sigma,s}$.
\item $f$ is the restriction to the real line of a function $F$ which is holomorphic in the strip $S_\sigma$ and satisfies $\sup_{\abs{y} < \sigma} \norm{F(x+iy)}_{H^s_x} < \infty$.
\end{enumerate}
\end{PWtheorem}

The proof given for $s=0$ in \cite[p.~209]{Katznelson1976} applies also for $s \in \R$ with some obvious modifications. 

\section{Main result}

Consider the Cauchy problem \eqref{DKG}, \eqref{Data} with data
\begin{equation}\label{GevreyData}
  (\psi_0,\phi_0,\phi_1) \in G^{\sigma_0,s}(\R;\C^2) \times G^{\sigma_0,r}(\R;\R) \times G^{\sigma_0,r-1}(\R;\R),
\end{equation}
where $\sigma_0 > 0$ and $(r,s) \in \R^2$. By the embedding \eqref{Inclusion1} and the existing well-posedness theory we know that this problem has a unique, smooth solution for all time, regardless of the values of $r$ and $s$. Our main result gives an algebraic lower bound on the radius of analyticity $\sigma(t)$ of the solution as the time $t$ tends to infinity.

\begin{theorem}\label{thm2} Let $\sigma_0 > 0$ and $(r,s) \in \R^2$. Then for any data \eqref{GevreyData} the solution of the Cauchy problem \eqref{DKG}, \eqref{Data} satisfies
\[
  (\psi,\phi,\partial_t \phi)(t) \in G^{\sigma(t),s} \times G^{\sigma(t),r} \times G^{\sigma(t),r-1}
  \quad \text{for all $t \in \R$},
\]
where the radius of analyticity $\sigma(t) > 0$ satisfies an asymptotic lower bound
\[
  \sigma(t) \ge \frac{c}{t^4}
  \qquad \text{as $\abs{t} \to \infty$},
\]
with a constant $c > 0$ depending on $m$, $M$, $\sigma_0$, $r$, $s$, and the norm of the data \eqref{GevreyData}. 
\end{theorem}

Observe that by the embedding \eqref{Inclusion2}, it suffices to prove Theorem \ref{thm2} for a single choice of $(r,s) \in \R^2$, and we choose $(r,s) = (1,0)$; global well-posedness in the Sobolev data space \eqref{SobolevData} at this regularity was first proved by Bournaveas \cite{Bournaveas:2000}. By time reversal it suffices to prove the theorem for $t > 0$, which we assume henceforth.

The first step in the proof is to show that in a short time interval $0 \le t \le \delta$, where $\delta > 0$ depends on the norm of the initial data, the radius of analyticity remains constant. This is proved by a contraction argument involving energy estimates, Sobolev embedding and a null form estimate which is somewhat similar to the one proved by Bournaveas in \cite{Bournaveas:2000}. Here we take care to optimize the dependence of $\delta$ on the data norms, since the local result will be iterated.

The next step is to improve the control of the growth of the solution in the time interval $[0,\delta]$, measured in the data norm \eqref{GevreyData}. To achieve this we show that, although the conservation of charge does not hold exactly in the Gevrey space $G^{\sigma,0}$, it does hold in an approximate sense. Iterating the local result we then obtain Theorem \ref{thm2}.

Analogous results for the KdV equation were proved by Bona, Gruji\'c and Kalisch in \cite{Bona2005}, but the method used there is quite different: They estimate in Gevrey-modified Bourgain spaces directly on any large time interval $[0,T]$ while we iterate a precise local result. While we have been able to adapt the method from \cite{Bona2005} to the DKG equations, this only gave us a rate $1/t^{8+}$, while our method gives $1/t^4$. The reason for this is twofold: (i) the contraction norms involve integration in time, which is a disadvantage when working on large time intervals, and (ii) it is not clear how to get any kind of approximate charge conservation when working directly on large time intervals. Our short-time iterative approach is more inspired by the ideas developed by Colliander, Holmer and Tzirakis in \cite{Colliander2008} in the context of global well-posedness in the standard Sobolev spaces\footnote{In particular, our argument also provides an alternative proof of the result of Bournaveas \cite{Bournaveas:2000} concerning global well-posedness for data $(\psi_0,\phi_0,\phi_1) \in L^2 \times H^1 \times L^2$, by setting $\sigma=0$ throughout.}, and the idea of almost conservation laws introduced in \cite{Colliander2002}.

Although our method gives a significantly better result for DKG than the method from \cite{Bona2005}, we do not know whether the result in Theorem \ref{thm2} is optimal. We do expect that the ideas introduced here can be applied also to other equations than DKG to obtain algebraic lower bounds on the radius of analyticity.

We now turn to the proofs. Leaving the case $m=0$ until the very end of the paper, we will assume $m > 0$ for now. By a rescaling we may assume $m=1$.

\section{Reformulation of the system}

It will be convenient to rewrite the DKG system as follows. Write $\psi= (\psi_+,\psi_-)^T$ and $\phi = \phi_+ + \phi_-$ with $\phi_\pm = \frac12 \left( \phi \pm i \angles{D_x}^{-1} \partial_t \phi \right)$, where $D_x = -i\partial_x$, hence $D_x$ and $\angles{D_x}$ are Fourier multipliers with symbols $\xi$ and $\angles{\xi} = (1+\abs{\xi}^2)^{1/2}$ respectively. For later use we note that $e^{\sigma \abs{D_x}}$ has symbol $e^{\sigma\abs{\xi}}$, while $e^{\pm \sigma D_x}$ has symbol $e^{\pm \sigma\xi}$.

Writing also $D_t = -i\partial_t$, the Cauchy problem \eqref{DKG}, \eqref{GevreyData} with $m=1$ is then equivalent to
\begin{equation}\label{DKGsplit}
\left\{
\begin{alignedat}{2}
  \left( D_t + D_x \right) \psi_+ &= -M\psi_- + \phi \psi_-,&
  \quad
  \psi_+(0) &= f_+ \in G^{\sigma_0,s},
  \\
  \left( D_t - D_x \right) \psi_- &= -M\psi_+ + \phi \psi_+,&
  \quad
  \psi_-(0) &= f_- \in G^{\sigma_0,s},
  \\
  \left( D_t + \angles{D_x} \right) \phi_+ &= - \angles{D_x}^{-1} \re \left( \overline{\psi_+} \psi_- \right),&
  \quad
  \phi_+(0) &= g_+ \in G^{\sigma_0, r},
  \\
  \left( D_t - \angles{D_x} \right) \phi_- &= + \angles{D_x}^{-1} \re \left( \overline{\psi_+} \psi_- \right),&
  \quad
  \phi_-(0) &= g_- \in G^{\sigma_0, r},
\end{alignedat}
\right.
\end{equation}
where $\psi_0 = (f_+,f_-)^T$ and $g_\pm = \frac12\left( \phi_0 \pm i \angles{D_x}^{-1} \phi_1 \right)$. We remark that $\overline{\phi_+} = \phi_-$, since $\phi$ is real-valued.

\section{Energy estimate}

Each line in \eqref{DKGsplit} is of the schematic form
\[
  \left( D_t + h(D_x) \right) u = F(t,x), \quad u(0,x) = f(x),
\]
with $h(\xi) = \pm \xi$ or $\pm\angles{\xi}$. Then for sufficiently regular $f$ and $F$ one has, by Duhamel's formula,
\[
  u(t) = W_{h(\xi)}(t) f + i \int_0^t W_{h(\xi)}(t-s) F(s) \, ds,
\]
where $W_{h(\xi)}(t) = e^{-ith(D_x)}$ is the solution group; it is the Fourier multiplier with symbol $e^{-ith(\xi)}$. From this one obtains immediately the following energy inequality, for any $\sigma \ge 0$ and $a \in \R$:
\begin{equation}\label{EnergyInequality}
  \norm{u(t)}_{G^{\sigma,a}} \le \norm{f}_{G^{\sigma,a}} + \int_0^t \norm{F(s)}_{G^{\sigma,a}} \, ds \qquad (t \ge 0).
\end{equation}

\section{Bilinear estimates}

We shall need the following null form estimate.

\begin{lemma}\label{NullLemma}
The solutions $u$ and $v$ of the Cauchy problems
\begin{alignat*}{2}
  (D_t+D_x) u &= F(t,x),& \qquad u(0,x) &= f(x),
  \\
  (D_t-D_x) v &= G(t,x),& \qquad v(0,x) &= g(x),
\end{alignat*}
satisfy
\[
  \norm{uv}_{L^2([0,T] \times \R)}
  \le C
  \left( \norm{f}_{L^2} + \int_0^T \norm{F(t)}_{L^2} \, dt \right)
  \left( \norm{g}_{L^2} + \int_0^T \norm{G(t)}_{L^2} \, dt \right)
\]
for all $T > 0$. Moreover, the same holds for $\overline u v$.
\end{lemma}

\begin{proof}
By Duhamel's formula, as in the proof of Theorem 2.2 in \cite{Klainerman1993}, one can reduce to the case $F=G=0$, and this case is easily proved by changing to characteristic coordinates or by using Plancherel's theorem as in \cite[Lemma 2]{Selberg:2008}. Replacing $u$ by its complex conjugate $\overline u$ does not affect the argument, since $u(t)=e^{-itD_x}f$ implies $\overline u = e^{-itD_x} \overline f$, as one can check on the Fourier transform side.
\end{proof}

\begin{corollary}\label{NullCorollary}
Let $\sigma \ge 0$. With notation as in Lemma \ref{NullLemma} we have
\begin{multline*}
  \norm{uv}_{L^2_t([0,T];G^{\sigma,0})}
  \le C
  \left( \norm{f}_{G^{\sigma,0}} + \int_0^T \norm{F(t)}_{G^{\sigma,0}} \, dt \right)
  \\
  \times\left( \norm{g}_{G^{\sigma,0}} + \int_0^T \norm{G(t)}_{G^{\sigma,0}} \, dt \right)
\end{multline*}
for all $T > 0$. Moreover, the same holds for $\overline u v$.
\end{corollary}

\begin{proof}
It suffices to prove that $\norm{e^{\pm \sigma D_x}(uv)}_{L^2_t([0,T];L^2)}$ is bounded by the right-hand side. But this follows from Lemma \ref{NullLemma}, since $e^{\pm \sigma D_x} (uv) = (e^{\pm \sigma D_x}u) (e^{\pm \sigma D_x}v)$, as is obvious on the Fourier transform side.
\end{proof}

We will also need the Sobolev product estimate
\begin{equation}\label{Sobolev}
  \norm{fg}_{G^{\sigma,0}} \le C \norm{f}_{G^{\sigma,1}} \norm{g}_{G^{\sigma,0}},
\end{equation}
where $C > 1$ is an absolute constant. For $\sigma=0$ this reduces to H\"older's inequality and the Sobolev embedding $H^1(\R) \subset L^\infty(\R)$, and the case $\sigma > 0$ is then deduced as in the proof of the corollary above.

\section{A local result}\label{LWP}

Here we prove the following local existence result:

\begin{theorem}\label{FirstStep} Let $\sigma_0 > 0$. For any
\[
  (\psi_0,\phi_0,\phi_1) \in X_0 := G^{\sigma_0,0}(\R;\C^2) \times G^{\sigma_0,1}(\R;\R) \times G^{\sigma_0,0}(\R;\R)
\]
there exists a time $\delta > 0$ such that the solution of the Cauchy problem \eqref{DKG}, \eqref{Data} satisfies $(\psi,\phi,\partial_t \phi) \in C\left([0,\delta]; X_0 \right)$. Moreover,
\begin{equation}\label{delta}
  \delta = \frac{c_0}{1 + a_0^2 + b_0},
\end{equation}
where $a_0 = \norm{f_+}_{G^{\sigma_0,0}} + \norm{f_-}_{G^{\sigma_0,0}}$, $b_0 = \norm{g_+}_{G^{\sigma_0,1}} + \norm{g_-}_{G^{\sigma_0,1}}$, and $c_0 > 0$ is a constant depending on the Dirac mass $M$.
\end{theorem}

\begin{proof} Consider the iterates $\psi_\pm^{(n)}$, $\phi_\pm^{(n)}$ given inductively by
\begin{gather*}
  \psi_\pm^{(0)}(t) = W_{\pm\xi}(t) f_\pm,
  \qquad
  \phi_\pm^{(0)}(t) = W_{\pm\angles{\xi}}(t) g_\pm,
  \\
  \begin{aligned}
  \psi_\pm^{(n+1)}(t) &= \psi_\pm^{(0)}(t)
  + i \int_0^t W_{\pm\xi}(t-s)\left( -M\psi_\mp^{(n)} + \phi^{(n)} \psi_\mp^{(n)} \right)(s) \, ds,
  \\
  \phi_\pm^{(n+1)}(t) &= \phi_\pm^{(0)}(t)
  \mp i \int_0^t  W_{\pm\angles{\xi}}(t-s) \angles{D_x}^{-1}
  \re \left( \overline{\psi_+^{(n)}} \psi_-^{(n)} \right)(s) \, ds
\end{aligned}
\end{gather*}
for $n \in \mathbb N_0$. Here $\phi^{(n)} = \phi_+^{(n)} + \phi_-^{(n)}$. Now set
\begin{align*}
  A_n(\delta) &= \sum_\pm \left( \norm{f_\pm}_{G^{\sigma_0,0}} + \int_0^\delta \norm{(D_t \pm D_x)\psi_\pm^{(n)}(t)}_{G^{\sigma_0,0}} \, dt \right),
  \\
  B_n(\delta) &= \sum_\pm \norm{\phi_\pm^{(n)}}_{L_t^\infty([0,\delta];G^{\sigma_0,1})}.
\end{align*}
We claim that, for $n \in \mathbb N_0$,
\begin{alignat}{2}
  \label{IterationEstimate1}
  A_0(\delta) &= a_0,& \qquad A_{n+1}(\delta) &\le a_0 + C \delta A_n(\delta) \bigl(M + B_n(\delta)\bigr),
  \\
  \label{IterationEstimate2}
  B_0(\delta) &= b_0,& \qquad B_{n+1}(\delta) &\le b_0 + C \delta^{1/2} A_n(\delta)^2,
\end{alignat}
where $C > 1$ is an absolute constant. By the energy inequality \eqref{EnergyInequality}, this reduces to
\begin{align*}
  \int_0^\delta \norm{-M\psi_\mp^{(n)}(t)}_{G^{\sigma_0,0}} \, dt
  &\le
  C\delta M A_n(\delta),
  \\
  \int_0^\delta \norm{\phi^{(n)}(t) \psi_\mp^{(n)}(t)}_{G^{\sigma_0,0}} \, dt
  &\le
  C \delta A_n(\delta) B_n(\delta),
  \\
  \int_0^\delta \norm{\overline{\psi_+^{(n)}(t)} \psi_-^{(n)}(t)}_{G^{\sigma_0,0}} \, dt
  &\le
  C \delta^{1/2} A_n(\delta)^2.
\end{align*}
The first estimate follows from the energy inequality \eqref{EnergyInequality}, the second from  \eqref{Sobolev}, and the third from Corollary \ref{NullCorollary}, after an application of H\"older's inequality in time. 

For convenience we replace $b_0$ in the right-hand side of \eqref{IterationEstimate2} by $a_0+b_0$. Then by induction we get $A_n(\delta) \le 2a_0$ and $B_n(\delta) \le 2(a_0+b_0)$ for all $n$, provided $\delta > 0$ is so small that $C \delta 2a_0 (M+2a_0+2b_0) \le a_0$ and $C \delta^{1/2} (2a_0)^2 \le a_0+b_0$, but the latter we replace by the more restrictive $C \delta^{1/2} 2a_0 2(a_0+b_0) \le a_0+b_0$, hence we get \eqref{delta}.

Applying the same estimates to
\begin{align*}
  \mathfrak A_n(\delta) &= \sum_\pm \int_0^\delta \norm{(D_t \pm D_x)(\psi_\pm^{(n)}-\psi_\pm^{(n-1)})(t)}_{G^{\sigma_0,0}} \, dt,
  \\
  \mathfrak B_n(\delta) &= \sum_\pm \norm{\phi_\pm^{(n)}-\phi_\pm^{(n-1)}}_{L_t^\infty([0,\delta];G^{\sigma_0,1})},
\end{align*}
one finds that
\begin{align}
  \label{IterationEstimate3}
  \mathfrak A_{n+1}(\delta) &\le C \delta \mathfrak A_n(\delta) \bigl(M + B_n(\delta)\bigr)
  + C \delta A_n(\delta) \mathfrak B_n(\delta),
  \\
  \label{IterationEstimate4}
  \mathfrak B_{n+1}(\delta) &\le 2C \delta^{1/2} A_n(\delta)\mathfrak A_n(\delta),
\end{align}
so taking $\delta$ even smaller, but still satisfying \eqref{delta}, then $\mathfrak A_{n+1}(\delta) \le \frac14 \mathfrak A_n(\delta) + \frac14 \mathfrak B_n(\delta)$ and $\mathfrak B_{n+1}(\delta) \le \frac14 \mathfrak A_n(\delta)$,
hence $\mathfrak A_{n+1}(\delta) + \mathfrak B_{n+1}(\delta) \le \frac12 \left( \mathfrak A_n(\delta) + \mathfrak B_n(\delta) \right)$. Therefore the iterates converge, and this concludes the proof of Theorem \ref{FirstStep}.
\end{proof}

\section{Growth estimate and almost conservation of charge}\label{SecondStepSection}

Next we estimate the growth in time of
\begin{align*}
  \mathfrak M_\sigma(t) &= \sum_{\epsilon \in \{-1,+1\}} \left( \norm{e^{\epsilon\sigma D_x}\psi_+(t)}_{L^2}^2 + \norm{e^{\epsilon\sigma D_x}\psi_-(t)}_{L^2}^2 \right),
  \\
  \mathfrak N_\sigma(t) &= \norm{\phi_+(t)}_{G^{\sigma,1}} + \norm{\phi_-(t)}_{G^{\sigma,1}},
\end{align*}
where $\sigma \in (0,\sigma_0]$ is considered a parameter. Note that $\mathfrak M_\sigma(t)$ is comparable to
\[
  \mathfrak M_\sigma'(t) = \norm{\psi_+(t)}_{G^{\sigma,0}}^2 + \norm{\psi_-(t)}_{G^{\sigma,0}}^2.
\]
By the local theory developed so far we know that $\mathfrak M_\sigma(t)$ and $\mathfrak N_\sigma(t)$ remain finite for times $t \in [0,\delta]$, where
\begin{equation}\label{Time}
  \delta = \delta(\sigma) = \frac{c_0}{1+\mathfrak M_\sigma(0) + \mathfrak N_{\sigma}(0)}.
\end{equation}
Moreover, from the proof of Theorem \ref{FirstStep} we know that
\begin{gather}
  \label{IterationBound1}
  \norm{f_\pm}_{G^{\sigma,0}} + \int_0^\delta \norm{(D_t \pm D_x)\psi_\pm(t)}_{G^{\sigma,0}} \, dt
  \le C\mathfrak M_\sigma(0)^{1/2},
  \\
  \label{IterationBound2}
  \norm{\phi_\pm}_{L_t^\infty([0,\delta];G^{\sigma,1})}
  \le C \left( \mathfrak M_\sigma(0)^{1/2} + \mathfrak N_\sigma(0) \right).
\end{gather}

We will prove the following.

\begin{theorem}\label{SecondStep}
With hypotheses as in Theorem \ref{FirstStep}, then for $\sigma \in (0,\sigma_0]$ and $\delta=\delta(\sigma)$ as in \eqref{Time}, we have
\begin{align}
  \label{Mest}
  \sup_{t \in [0,\delta]} \mathfrak M_\sigma(t)
  &\le
  \mathfrak M_\sigma(0) + C\sigma \delta^{1/2} \mathfrak M_\sigma(0) \left( \mathfrak M_\sigma(0)^{1/2} + \mathfrak N_\sigma(0) \right),
  \\
  \label{Nest}
  \sup_{t \in [0,\delta]} \mathfrak N_\sigma(t)
  &\le
  \mathfrak N_\sigma(0) + C \delta^{1/2} \mathfrak M_\sigma(0),
\end{align}
where $C > 1$ is an absolute constant.
\end{theorem}

It suffices to prove these estimates at the endpoint $t=\delta$. 

The estimate for $\mathfrak N_\sigma$ follows from the energy inequality and Corollary \ref{NullCorollary} as in the proof of \eqref{IterationEstimate2}, and taking into account \eqref{IterationBound1}. 

To estimate $\mathfrak M_\sigma$ we proceed as in the proof of conservation of charge. Let $\epsilon \in \{ -1, +1\}$ and write
\[
  \mathfrak M_{\sigma,\epsilon}(t) = \norm{e^{\epsilon\sigma D_x}\psi_+(t)}_{L^2}^2 + \norm{e^{\epsilon\sigma D_x}\psi_-(t)}_{L^2}^2,
\]
so that $\mathfrak M_\sigma = \mathfrak M_{\sigma,-1} + \mathfrak M_{\sigma,+1}$. Applying $e^{\epsilon\sigma D_x}$ to each side of the Dirac equations in \eqref{DKGsplit} gives\footnote{Observe that $e^{\epsilon\sigma D_x}(fg) = (e^{\epsilon\sigma D_x}f)(e^{\epsilon\sigma D_x}g)$, as is obvious on the Fourier transform side. This is the reason for using the norm $\mathfrak M_\sigma$ instead of the simpler $\mathfrak M_\sigma'$.}
\[
  (D_t \pm D_x)\Psi_\pm = (\phi-M) \Psi_\mp + F_\mp,
\]
where
\[
  \Psi_\pm = e^{\epsilon\sigma D_x}\psi_\pm,
  \qquad
  F_\pm = (e^{\epsilon\sigma D_x}\phi-\phi) \Psi_\pm.
\]
Since $\phi$ and $M$ are real-valued we then get
\begin{align*}
  \frac{d}{dt} \mathfrak M_{\sigma,\epsilon}(t)
  &= 2\re \int \partial_t\Psi_+(t,x) \overline{\Psi_+(t,x)} + \partial_t\Psi_-(t,x) \overline{\Psi_-(t,x)} \, dx
  \\
  &= \int \partial_x\left(-\Psi_+(t,x) \overline{\Psi_+(t,x)} + \Psi_-(t,x) \overline{\Psi_-(t,x)}\right) \, dx
  \\
  &\qquad+ 2\re \int i(\phi(t,x)-M) \left( \Psi_-(t,x) \overline{\Psi_+(t,x)} + \Psi_+(t,x) \overline{\Psi_-(t,x)} \right) \, dx
  \\
  &\qquad+ 2\re \int iF_-(t,x) \overline{\Psi_+(t,x)} + iF_+(t,x) \overline{\Psi_-(t,x)} \, dx
  \\
  &= - 2\im \int F_-(t,x) \overline{\Psi_+(t,x)} + F_+(t,x) \overline{\Psi_-(t,x)} \, dx.
\end{align*}
In the last step we used the fact that we may assume that $\Psi_\pm(t,x)$ decays to zero at spatial infinity. Indeed, if we want to prove the estimates in Theorem \ref{SecondStep} for a given $\sigma$, then by the monotone convergence theorem it suffices to prove it for all $\sigma' < \sigma$, and then we get the decay by the Riemann-Lebesgue lemma.

Integration in time yields
\begin{align*}
  \mathfrak M_{\sigma,\epsilon}(\delta)
  &=
  \mathfrak M_{\sigma,\epsilon}(0) - 2\im \int_0^\delta \int \left( F_-(t,x)\overline{\Psi_+(t,x)}
  + F_+(t,x)\overline{\Psi_-(t,x)} \right) \, dx \, dt
  \\
  &\le \mathfrak M_{\sigma,\epsilon}(0)
  + 4 \delta^{1/2} \norm{(e^{\epsilon\sigma D_x}-1)\phi}_{L_t^\infty([0,\delta];L^2)}
  \norm{\overline{\Psi_+}\Psi_-}_{L^2([0,\delta] \times \R)}
  \\
  &\le \mathfrak M_{\sigma,\epsilon}(0)
  + 4 \delta^{1/2} \sigma \norm{\abs{D_x}\phi}_{L_t^\infty([0,\delta];G^{\sigma,0})} \norm{\overline{\Psi_+}\Psi_-}_{L^2([0,\delta] \times \R)}
\end{align*}
where we applied H\"older's inequality and the symbol estimate
\[
  \abs{e^{\epsilon\sigma \xi} - 1} \le \sigma\abs{\xi} e^{\sigma\abs{\xi}}.
\]
Applying Lemma \ref{NullLemma} and taking into account the bounds \eqref{IterationBound1} and \eqref{IterationBound2} we then obtain \eqref{Mest}, and this concludes the proof of Theorem \ref{SecondStep}.

\section{Conclusion of the proof}

All the tools required to complete the proof of the main result, Theorem \ref{thm2}, are now at hand.

We are given $\sigma_0 > 0$ and data such that $\mathfrak M_{\sigma_0}(0)$ and $\mathfrak N_{\sigma_0}(0)$ are finite. Our task is to prove that for all large $T$, the solution has a positive radius of analyticity
\begin{equation}\label{Goal}
  \sigma(t) \ge \frac{c}{T^4} \quad \text{for all $t \in [0,T]$,}
\end{equation}
where $c > 0$ is a constant depending on the Dirac mass $M$ and the data norms $\mathfrak M_{\sigma_0}(0)$ and $\mathfrak N_{\sigma_0}(0)$.

Since we are interested in the behaviour as $T \to \infty$, we may certainly assume
\begin{equation}\label{Tlarge}
  \mathfrak M_{\sigma_0}(0)+\mathfrak N_{\sigma_0}(0) \le T^2.
\end{equation}
Now fix such a $T$ and let $\sigma \in (0,\sigma_0]$ be a parameter to be chosen; then \eqref{Tlarge} of course holds also with $\sigma_0$ replaced by $\sigma$.

Let $A \gg 1$ denote a constant which may depend on $M$, $\mathfrak M_{\sigma_0}(0)$ and $\mathfrak N_{\sigma_0}(0)$; the choice of $A$ will be made explicit below. 

As long as $\mathfrak M_\sigma(t) \le 2\mathfrak M_\sigma(0)$ and $\mathfrak N_\sigma(t) \le 2AT^2$ we can now apply the local results, Theorems \ref{FirstStep} and \ref{SecondStep}, with a uniform time step
\begin{equation}\label{TimeStep}
  \delta = \frac{c_0}{A T^2},
\end{equation}
where $c_0 > 0$ depends only on $M$. We can choose $c_0$ so that $T/\delta$ is an integer. Proceeding inductively we cover intervals $[(n-1)\delta,n\delta]$ for $n=1,2,\dots$, obtaining
\begin{align*}
  \mathfrak M_\sigma(n\delta)
  &\le
  \mathfrak M_\sigma(0) + nC\sigma \delta^{1/2} (2\mathfrak M_\sigma(0)) (4AT^2),
  \\
  \mathfrak N_\sigma(n\delta)
  &\le
  \mathfrak N_\sigma(0) + nC\delta^{1/2} 2\mathfrak M_\sigma(0),
\end{align*}
and in order to reach the prescribed time $T=n\delta$ we require that
\begin{align}
  \label{final1}
  nC\sigma \delta^{1/2} (2\mathfrak M_\sigma(0)) (4AT^2) &\le \mathfrak M_\sigma(0),
  \\
  \label{final2}
  nC\delta^{1/2} 2\mathfrak M_\sigma(0) &\le AT^2.
\end{align}
From $T=n\delta$ and \eqref{TimeStep} we get $n\delta^{1/2} = T \delta^{-1/2} = T c_2 A^{1/2} T$, where $c_2$ depends on  $M$. Therefore \eqref{final1} and \eqref{final2} reduce to
\begin{align}
  \label{final3}
  C\sigma c_2 A^{1/2} T^2 2(4AT^2) &\le 1,
  \\
  \label{final4}
  C c_2 A^{1/2} T^2 2\mathfrak M_\sigma(0) &\le AT^2.
\end{align}
To satisfy the latter we choose $A$ so large that
\[
 Cc_2 2\mathfrak M_\sigma(0) \le A^{1/2}.
\]
Finally, \eqref{final3} is satisfied if we take $\sigma$ equal to the right-hand side of \eqref{Goal}. This concludes the proof of our main result, Theorem \ref{thm2}, in the case $m > 0$.

\section{The case $m=0$}

Set $m=0$. The conclusion in Theorem \ref{FirstStep} remains true. To see this, we add $\phi$ to each side of the second equation in \eqref{DKG}, hence the last two equations in \eqref{DKGsplit} are replaced by
\[
  \left( D_t \pm \angles{D_x} \right) \phi_\pm = \mp \angles{D_x}^{-1} \left( \phi + \re\left(\overline{\psi_+} \psi_- \right)\right),
\]
where as before $\phi = \phi_+ + \phi_-$ with $\phi_\pm = \frac12 \left( \phi \pm i \angles{D_x}^{-1} \partial_t \phi \right)$. Then the proof of Theorem~\ref{FirstStep} goes through with some obvious changes: in \eqref{IterationEstimate2} there appears an extra term $\delta B_n(\delta)$ on the right-hand side, and in \eqref{IterationEstimate4} a term $\delta \mathfrak B_n(\delta)$.

Next we consider the changes that need to be made in the argument from Section \ref{SecondStepSection}. Since Theorem \ref{FirstStep} remains unchanged, then so does \eqref{Time}--\eqref{IterationBound2}. Now observe that the norm
\[
  \mathfrak N_\sigma(t) = \norm{\phi_+(t)}_{G^{\sigma,1}} + \norm{\phi_-(t)}_{G^{\sigma,1}}
\]
is equivalent to
\[
  \norm{\angles{D_x}\phi(t)}_{G^{\sigma,0}} + \norm{\partial_t \phi(t)}_{G^{\sigma,0}},
\]
which in turn is equivalent to
\[
  \norm{\phi(t)}_{L^2} + \norm{\abs{D_x}\phi(t)}_{G^{\sigma,0}} + \norm{\partial_t \phi(t)}_{G^{\sigma,0}}.
\]
But if we write $\phi = \Phi_+ + \Phi_-$ with $\Phi_\pm = \frac12 \left( \phi \pm i \abs{D_x}^{-1} \partial_t \phi \right)$, then the norm
\[
  \norm{\abs{D_x}\phi(t)}_{G^{\sigma,0}} + \norm{\partial_t \phi(t)}_{G^{\sigma,0}}
\]
is equivalent to
\[
  \mathfrak N_\sigma'(t) = \norm{\abs{D_x}\Phi_+(t)}_{G^{\sigma,0}} + \norm{\abs{D_x}\Phi_-(t)}_{G^{\sigma,0}}.
\]
Thus we have the norm equivalence
\begin{equation}\label{NormEquivalence}
  \mathfrak N_\sigma(t) \sim \norm{\phi(t)}_{L^2} + \mathfrak N_\sigma'(t),
\end{equation}
where the implicit constants are independent of $t$, of course. In particular, this means that \eqref{Time} can be replaced by
\begin{equation}\label{NewTime}
  \delta = \frac{c_0}{1 + \mathfrak M_{\sigma}(0) + \norm{\phi(0)}_{L^2} + \mathfrak N_\sigma'(0)}.
\end{equation}

The first term on the right-hand side of \eqref{NormEquivalence} is a priori under control. Indeed, from the energy inequality for the wave operator $\square = -\partial_t^2+\partial_x^2$ and by conservation of charge we have
\begin{equation}\label{Growth}
  \norm{\phi(t)}_{L^2} = O(t^2)
\end{equation}
as $t \to \infty$, where the implicit constant depends on $\norm{(\phi_0,\phi_1)}_{L^2 \times H^{-1}} + \norm{\psi_0}_{L^2}^2$.

Noting that
\[
  \left( D_t \pm \abs{D_x} \right) \Phi_\pm = \mp \abs{D_x}^{-1} \re\left(\overline{\psi_+} \psi_- \right),
\]
the argument used to prove Theorem \ref{SecondStep} now gives
\begin{align*}
  \sup_{t \in [0,\delta]} \mathfrak M_\sigma(t)
  &\le
  \mathfrak M_\sigma(0) + C\sigma \delta^{1/2} \mathfrak M_\sigma(0) \left( \mathfrak M_\sigma(0)^{1/2} + \norm{\phi(0)}_{L^2} + \mathfrak N_\sigma'(0) \right),
  \\
  \sup_{t \in [0,\delta]} \mathfrak N_\sigma'(t)
  &\le
  \mathfrak N_\sigma'(0) + C \delta^{1/2} \mathfrak M_\sigma(0).
\end{align*}
Using these as well as \eqref{NewTime} and \eqref{Growth}, the argument from the previous section goes through to show that $\sigma(t) \ge c/t^4$ as $t \to \infty$. This concludes the case $m=0$.

\section*{Acknowledgements} The authors are indebted to Hartmut Pecher and to an anonymous referee for helpful comments on an earlier version of this article. Sigmund Selberg was supported by the Research Council of Norway, grant no.~213474/F20. Achenef Tesfahun acknowledges support from the German Research Foundation, Collaborative Research Center 701.

\bibliographystyle{amsplain} 
\bibliography{DKGbibliography}

\end{document}